\documentclass[12pt]{amsart}

\usepackage[hmargin=0.8in,height=8.6in]{geometry}
\usepackage{amssymb,amsthm, times, enumerate}
\usepackage{delarray,verbatim}

\usepackage{ifpdf}
\ifpdf
\usepackage[pdftex]{graphicx}
 \else
\usepackage[dvips]{graphicx}
 \fi

\usepackage{ifpdf}
\usepackage{color}
\definecolor{webgreen}{rgb}{0,.5,0}
\definecolor{webbrown}{rgb}{.8,0,0}
\definecolor{emphcolor}{rgb}{0.95,0.95,0.95}

\usepackage{hyperref}
\hypersetup{%
	colorlinks=true,
	linkcolor=webbrown,
	filecolor=webbrown,
	citecolor=webgreen,
	breaklinks=true}
\ifpdf \hypersetup{pdftex,
	pdfstartview=FitH, 
	bookmarksopen=true,
	bookmarksnumbered=true
} \else \hypersetup{dvips} \fi

\newcommand{\J}{\mathcal{J}}

\newcommand {\e}{\mathbb{E}}

\numberwithin{equation}{section}

\newtheorem{theorem}{Theorem}[section]
\newtheorem{proposition}{Proposition}[section]

\newtheorem{remark}{Remark}[section]
\newtheorem{lemma}{Lemma}[section]

\newtheorem{assump}{Assumption}[section]

\numberwithin{remark}{section} \numberwithin{proposition}{section}
\numberwithin{corollary}{section}
\newcommand {\R}{\mathbb{R}}

\newcommand {\p}{\mathbb{P}}

\newcommand {\E}{\mathbb{E}}

\newcommand{\diff}{{\rm d}}

\newcommand{\lev}{L\'{e}vy }

\title{L\'evy bandits under Poissonian decision times}

	\author[J. L. P\'erez]{Jos\'e-Luis P\'erez$^*$}

		\thanks{$*$\, Department of Probability and Statistics, Centro de Investigaci\'on en Matem\'aticasA.C. Calle Jalisco
		s/n. C.P. 36240, Guanajuato, Mexico. Email: jluis.garmendia@cimat.mx}

	\author[K. Yamazaki]{Kazutoshi Yamazaki$^\ddagger$}
\thanks{$\ddagger$\, School of 
Mathematics and Physics, The University of Queensland, St Lucia,
Brisbane, QLD 4072, Australia. Email: k.yamazaki@uq.edu.au}

\begin{document}

	\begin{abstract} 
	We consider a version of the continuous-time multi-armed bandit problem where decision opportunities arrive  at Poisson arrival times, and study its Gittins index policy.  When driven by spectrally one-sided \lev processes, the Gittins index can be written explicitly in terms of the scale function, and is shown to converge to that in the classical \lev bandit of Kaspi and Mandelbaum \cite{KM95}.
		\ \ 
		\\
		\noindent \small{\noindent  AMS 2020 Subject Classifications: 60G51, 93E20, 90B36 \\
			\textbf{Keywords:} multi-armed bandits, Gittins index, \lev processes, stochastic control}
	\end{abstract}

\maketitle

\section{Introduction}
The multi-armed bandit (MAB) problem is a stochastic dynamic optimization problem where at each decision time, one of  $J$ alternatives (called \emph{arms}) is selected and a reward is then collected. In its   infinite-time horizon formulation, the objective is to derive a policy (strategy) that maximizes the expected total discounted rewards. 
An important assumption of the MAB is that the state of each arm changes only when it is selected and the evolution of each arm is independent of others. By exploiting these features, the existing results have shown the optimality of the so-called \emph{Gittins index} policy \cite{Gittins}, which selects at each time the arm with the highest Gittins index. The Gittins index is defined in terms of an optimal stopping problem and can be computed separately for each arm. The state of each arm is usually assumed to evolve in a Markovian fashion, and thus the MAB can be seen as an instance of the Markov decision process (MDP) problem (see \cite{Puterman}). This optimality result is important in the MDP research in that the original $J$-dimensional MDP  problem can be reduced to $J$ independent $1$-dimensional MDPs, providing a way to overcome the \emph{curse of dimensionality}.

The continuous-time MAB theory has been developed in the 1990s by, for example, El Karoui and Karatzas \cite{ElKaroui_Karatzas} and Kaspi and Mandelbaum \cite{KM95, Mandelbaum88}.
 For the setting where an arm is continuously selected, advanced stochastic analysis techniques, in particular local time analysis, are required to precisely model the problem. However,  main results in the discrete-time model continue to hold in the continuous-time model.  The Gittins index can be defined in a similar way, and  it is optimal to choose the arm with the highest Gittins index at all times.
  The advantage of studying a continuous-time model 
  is that  the Gittins index can be sometimes obtained in an explicit form while numerical approaches are usually required in the discrete time case.
 In particular, Kaspi and Mandelbaum \cite{KM95} considered the model driven by one-dimensional \lev processes and expressed the Gittins index concisely using the Wiener-Hopf factorization and excursion theory of \lev processes.

In this paper, we revisit the MAB problem, in particular the \lev model of \cite{KM95}, and study a version 
where
decision opportunities arrive only at independent Poisson arrival times.  This problem is at the interface between the discrete-time and continuous-time models. The state of  each arm evolves continuously whereas decision times are discrete. 
We consider two variations of the problem and study the corresponding Gittins index policy for each. We introduce cutting-edge results of the fluctuation theory of \lev processes (e.g.\ \cite{Be, K}), which had not been available when the results of \cite{KM95} were developed.
%
%

The Poissonian observation model has recently drawn much attention in the stochastic control literature (see, e.g., \cite{Avanzi_Cheung_Wong_Woo,   Czarna,PPY,  PPSY,Perez_Yamazaki_Bensoussan,ZCY}), but it has not been considered in the MAB problem. 
With the assumption that decision opportunities arrive at exponential times, an analytical approach is still possible thanks to recently developed fluctuation-theory-based methods such as \cite{Albrecher, Lkabous, PY}.
The Gittins index can be written in terms of the so-called \emph{scale function} when the \lev process has only one-sided jumps (spectrally one-sided). We further investigate its connection with \cite{KM95} by showing the convergence of the Gittins index to that obtained in \cite{KM95}. 

The rest of the paper is organized as follows. In Section \ref{section_preliminaries}, we review the classical discrete-time and continuous-time models of the MAB problem, as well as the \lev model of \cite{KM95}. In Section \ref{section_problem}, we introduce the Poissonian decision time models and propose our Gittins index policy. In Section \ref{section_SN}, we review the
fluctuation theory of spectrally one-sided \lev processes and the scale function, and then obtain the Gittins index explicitly in terms of the scale function. In Section \ref{section_convergence}, we
 show the convergence of our Gittins index to that in  \cite{KM95}.  We conclude the paper in Section \ref{section_conclusion}. 

\section{Preliminaries} \label{section_preliminaries}


\subsection{Discrete-time Markovian bandits} \label{section_discrete_time}

A (multi-armed) Markovian bandit consists of $J \geq 1$ arms. For each arm labeled $j \in \mathcal{J} := \{1, \ldots, J\}$, we use an $E^j$-valued discrete-time Markov process $X^j = (X^j(k))_{k \geq 0}$ defined on a probability space $(\Omega^j, \mathcal{F}^j, \p^j)$ to model the state dynamics of arm $j$. These $J$ processes are assumed  mutually independent.
At each time $k \geq 0$, one of the $J$ arms is selected. A state-dependent reward  
is then collected and the state of the arm changes. On the other hand, the states of all the other arms remain the same.  

A policy (strategy) $\pi = (\pi (k))_{k \geq 0}$ models the arm to be selected at each time $k \geq 0$. It is adapted in the sense that $\pi(k)$ must be determined based only on the information collected until time $k$. 
%
%
%
Let $\sigma^\pi_j(k)$ be the number of times arm $j \in \mathcal{J}$ has been selected before time $k \geq 0$ under policy $\pi$.  At time $k$, immediately before making the decision, the state of the arms are $(X^{j}(\sigma_j^\pi(k)), j \in \mathcal{J})$. Using these, the decision maker needs to determine $\pi(k)$ and the reward from arm $\pi(k)$  
is collected. 
The counters $(\sigma^\pi_j; j \in \J)$ are updated as follows:
$\sigma^\pi_j(k+1) = \sigma^\pi_j(k)+1$ if $j = \pi(k)$ while $\sigma^\pi_j(k+1) = \sigma^\pi_j(k)$ if $j \neq \pi(k)$.

The objective of the MAB problem is to obtain an optimal policy that maximizes the total expected reward over an infinite-time horizon. Given a discount factor $0 < \beta < 1$ and a reward function $R^j: E^j \to \R$ for each $j \in \mathcal{J}$, the problem is to compute the value function
\begin{align}
v(\mathbf{x}) = \sup_{\pi}  \E_{\mathbf{x}}\left[\sum_{k=0}^\infty \beta^k 
 R^{\pi(k)}(X^{\pi(k)}(\sigma^\pi_{\pi(k)}(k)))
  \right], \label{discrete_time_model}
\end{align}
for $\mathbf{x} = (x^1, \ldots, x^J) \in E^1 \times \cdots \times E^J$ and obtain an optimal strategy $\pi^*$. Here,  $\E_{\mathbf{x}}$ is the conditional expectation when $X^j(0) = x^j$ for $j \in \J$.

In this problem, there exists a simple policy, known as the Gittins index
policy, that works optimally in every instance
of the MAB problem. 
It is optimal to choose the arm with the highest Gittins index, i.e.\ to select, at each $k \geq 0$,
\begin{align*}
\pi^*(k) = \arg \max_{j \in \mathcal{J}} \Gamma_j (X^j(\sigma_j^\pi(k))),
\end{align*}
where the Gittins index of arm $j$ at state $x$ is a ``normalized maximal reward''
\begin{align}
\Gamma_j(x) := \sup_{\tau \geq 1} \frac {\E_{x}^j \left[ 
\sum_{k=0}^{\tau-1} \beta^k R^j(X^j(k)) \right]}
{\E_{x}^j \left[  \sum_{k=0}^{\tau-1} \beta^k \right]}, \label{gittins_index_discrete_time}
\end{align}
which can be computed without any knowledge about the other arms. Here, $\tau$ is chosen over all stopping times greater than or equal to $1$ with respect to the filtration generated by
$(X^j(k))_{k \geq 0}$, and $\E^j_x$ is the expectation operator of the law of $\p^j_x$ under which $X^j(0) = x$ \emph{assuming arm $j$ is always selected}.

The Gittins index \eqref{gittins_index_discrete_time} can be derived by computing the unique value of $\gamma$ such that the value of the optimal stopping problem 
\begin{align}
\sup_{\tau \geq 1 }\E_x \Big[ \sum_{k=0}^{\tau-1} \beta^k \left(R^j(X^j(k)) -\gamma\right)\Big] \label{opt_stopping_discrete}
\end{align} becomes zero; see \cite[Eqn.\ (20)]{Mandelbaum88}.

\subsection{Continuous-time bandits} \label{subsection_cont_time_bandit}
In the continuous-time model, the state process $X^j = (X^j(t))_{t \geq 0}$, for  $j \in \J$,  is an $E^j$-valued continuous-time process defined on a probability space $(\Omega^j, \mathcal{F}^j, \p^j)$. Again, these processes are assumed mutually independent.

The evolution of the arms is given by 
\[
\mathbf{X} := (\mathbf{X}(\mathbf{s}) = (X^1(s_1),\ldots, X^J(s_J)), \mathbf{s} = (s_1, \ldots, s_J) \geq \mathbf{0}),
\]
which is adapted to the multiparameter filtration $(\mathcal{F}(\mathbf{s}) = \mathcal{F}^1(s_1) \vee \cdots \vee \mathcal{F}^J(s_J))_{\mathbf{s} \geq 0}$ where $(\mathcal{F}^j(t))_{t \geq 0}$ is the natural filtration generated by $X^j$ for $j \in \mathcal{J}$.  For more details, we refer the reader to \cite{KM95}.

In the continuous-time model, a strategy is given by a $J$-dimensional process 
\[
\mathbf{T} = (\mathbf{T}(t) = (T^1(t), \ldots, T^J(t)); t \geq 0)
\] 
where $T^j(t)$ is the amount of time allocated to arm $j$ until $t$. Naturally,  $\mathbf{T}(t)$ is nondecreasing, and
$T^1(t) + \cdots + T^J(t) = t$ for all $t \geq 0$. A strategy has to be adapted and hence
\[
\left\{ T^1(t) \leq s_1, \ldots, T^J(t) \leq s_J\right\} \in \mathcal{F}(\mathbf{s})
\]
for all $t \geq 0$ and $\textbf{s} = (s_1, \ldots, s_J) \in [0,\infty)^J$.

Given a discount factor $q > 0$ and a reward function $r^j:E^j \to \R$, $j \in \J$, the objective is to maximize the total discounted expected reward
\[
v(\mathbf{x}) = \sup_{\mathbf{T}} \E_{\mathbf{x}} \left[ \int_0^\infty e^{-q t} \sum_{j \in \mathcal{J}} r^j \left(X^j(T^j(t)) \right) \diff  T^j(t)  
\right]
\]
over all allocation strategies $\mathbf{T}$. Regarding the reward function, for the problem to be well-defined, it is assumed that
\begin{align}
\E_x^j \left[ \int_0^\infty e^{-q t} |r^j (X^j(t))| \diff t \right] < \infty, \quad j \in \mathcal{J}, \; x \in E_j \label{integrability_r}
\end{align}
where again $\E^j_x$ is the expectation operator of the law $\p^j_x$ of $X^j$ under which $X^j(0) = x$ assuming arm $j$ is always selected.

The Gittins index in the continuous-time model is analogous to \eqref{gittins_index_discrete_time} where the summations are replaced by integrals:
\begin{align}
\Gamma_j(x) := \sup_{\tau > 0} \frac {\E_{x}^j \left[ 
\int_0^\tau e^{-qt} r^j(X^j(t)) \diff t \right]}
{\E_{x}^j \left[ \int_0^\tau e^{-qt} \diff t \right]}, \quad j \in \J, \; x \in E_j. \label{gittins_index_continuous_time}
\end{align}
Similar to the discrete-time counterpart,  it is optimal to choose the arm with the highest Gittins index, although some technical details are needed to deal with the times at which multiple Gittins indices coincide.

\subsection{\lev bandits} As a special case of the continuous-time model, in Kaspi and Mandelbaum \cite{KM95}, an explicit expression of the Gittins index \eqref{gittins_index_continuous_time} was obtained under the following assumption: 
for each $j \in \J$,
\begin{enumerate}
\item the state process $X^j$ is an $\R$-valued \lev process ($E_j = \R$);
\item the reward function $r^j$ is increasing and continuous. 
\end{enumerate}


%
Under these assumptions, the Gittins index \eqref{gittins_index_continuous_time} for arm $j \in \J$  when its state is $x \in \R$ can be written concisely by
%
\[
\Gamma^j (x) = \int_{[0,\infty)} r^j (x + y) \mu^j(\diff y), \quad x \in \R,
\]
where $\mu^j$ is the probability measure determined by
\begin{align}  \label{mu_laplace2}
\int_{[0,\infty)} e^{-\theta y} \mu^j (\diff y) = \frac {\bar{\varphi}^j (q,0)} {\bar{\varphi}^j (q,\theta)}, \quad \theta \geq 0,
\end{align}
in terms of the Laplace exponent $\bar{\varphi}^j (q,\theta)$ of the inverse local time at maximum 
and the ascending ladder height process $(L^{-1}(t), H(t))_{t \geq 0}$
(see Section \ref{section_classical_case_SN}). The above characterization in terms of \eqref{mu_laplace2} is given in (4.5) of \cite{KM95}, which is derived from an alternative expression (3.3) of \cite{KM95} via the Wiener-Hopf factorization.


The optimal value function can be written in terms of the Gittins index function. We have
\[
v(\mathbf{x}) = \E_{\mathbf{x}} \left[ 
 \int_0^\infty e^{-q t} \bigvee_{j \in \J} \Gamma^j [\underline{X}^j(T^{*,j}(t))] \diff t  \right]
\]
where $\underline{X}^j(t) := \inf_{0 \leq u \leq t} X^j(u)$ is the running infimum process and $(T^{*,j}; j \in \J)$ is the allocation time under the Gittins index policy.

%

\section{Periodic observation models} \label{section_problem}

We now consider a version of the continuous-time model where decision times arrive at exponential times. 
Let $N = (N(t))_{t \geq 0}$ be a Poisson process with rate $\lambda > 0$, independent of $X^j$ for all $j \in  \J$, with arrival times  $\mathcal{T} := \{T_1,T_2,\ldots \}$ (hence $T_{k+1} - T_{k}$, $k \geq 0$, is an independent exponential time with  mean $\lambda^{-1}$ where $T_0 := 0$ for convenience). We assume the decision times are given by $\mathcal{T} \cup \{0\}$.  For each $T_k$, $k \geq 0$, one arm is selected. If arm $j$ is selected at $T_k$, its state changes as a \lev process according to the law $\p^j$ until the next decision time $T_{k+1}$, while the states of the other arms remain the same during the period.

We consider two variations of the problem
when $X^j$ is a one-dimensional \lev process  for each $j \in \mathcal{J}$.
 
 \subsection{Problem 1}
 With a discount factor $q > 0$ and a reward function $R^j: \R \to \R$, we consider the maximization problem:
\begin{align*}
v(\mathbf{x}) =
\sup_{\pi} \E_{\mathbf{x}}\left[  \sum_{k=0}^\infty e^{-q T_k}  
 R^{\pi(k)}(X^{\pi(k)}(\sigma^\pi_{\pi(k)}(k)))  \right],
\end{align*}
which can be seen as a version of the discrete-time model \eqref{discrete_time_model} where the deterministic decision times $\mathbb{N} \cup \{0\}$ are replaced by the Poisson arrival times (in addition to zero) $\mathcal{T} \cup \{0\}$. 
Here, analogously to \eqref{discrete_time_model}, $\pi(k)  \in \J$ denotes the arm selected at time $T_k$ and  
\[
\sigma^\pi_j(k) := \sum_{i=0}^{k-1} (T_{i+1} - T_i) 1_{\{
 \pi(i) = j
\}}\]
 is the amount of times arm $j$ has been active before $T_k$ under policy $\pi$.
The selection $\pi(k)$ must be determined based only on the information collected right before making decisions, i.e.\, 
\begin{align}
T_1, \ldots, T_k \quad \textrm{and} \quad X^j(t), t \leq \sigma^\pi_j(k) \; \textrm{for all} \; j \in \J.
\label{filtration_temp}
\end{align}



Analogous to what is assumed in \cite{KM95}, we assume the following.
\begin{assump} \label{assump_R_increasing}
For each $j \in \J$, we assume
\begin{enumerate}
\item the reward function $R^j$ is strictly increasing,
\item We have
$\E_x^j \Big[ \sum_{k=0}^\infty e^{-q T_k} |R^j(X^j(T_k)| \Big]<\infty$ for  $x\in\R$,
where $\E_x^j$ is the same as that in the continuous-time model in Section \ref{subsection_cont_time_bandit}.
\end{enumerate}
\end{assump}

\subsection{Problem 2}

As a special case of Problem 1, we also consider a variant of the continuous-time model reviewed in Section \ref{subsection_cont_time_bandit}, where the change of selection can be made only at $\mathcal{T} \cup \{0\}$. At each $T_{k}$, an arm is selected and it is locked in until $T_{k+1}$.  The difference from Problem 1 is that the reward depends on the path between the decision times. This can be modeled as a maximization problem
\[
v(\mathbf{x}) = \sup_{\pi}  v_\pi(\mathbf{x}),
\]
with
\[
v_\pi(\mathbf{x}) =\E_{\mathbf{x}} \left[ \int_0^\infty e^{-q t} \sum_{j \in \mathcal{J}} r^j \left(X^j(S_j^\pi(t)) \right) \diff  S_j^\pi(t)  \right],
\]
where, for each $j \in \mathcal{J}$, $r^j: \R \to \R$ and $S_j^\pi(t)$ is the amount of time at which $j$ has been active until time $t$ under $\pi$. For $T_k \leq t < T_{k+1}$, $S_{\pi(k)}^\pi(t) = 
\sigma^\pi_{\pi(k)}(k)
+(t-T_k)$, and hence
%
%
%
\[
v_\pi(\mathbf{x}) = \E_{\mathbf{x}} \left[  \sum_{k=0}^\infty \int_{T_{k}}^{T_{k+1}} e^{-qt} r^{\pi(k)}\left(X^{\pi(k)}(\sigma^\pi_{\pi(k)}(k)+(t-T_k)) \right) \diff t   \right].
\]
With $\mathcal{G}(T_k)$ all information observed until $T_k$ under $\pi$ (i.e.\  \eqref{filtration_temp}) and $\textrm{e}_\lambda$ an independent exponential random variable with mean $\lambda^{-1}$,
noting that $\pi(k) \in \mathcal{G}(T_k)$ and because $T_{k+1}-T_k \sim \textrm{e}_\lambda$ and is independent of $\mathcal{G}(T_k)$, the strong Markov property gives
\begin{align} \label{v_pi}
\begin{split}
v_\pi(\mathbf{x}) &= \E_{\mathbf{x}} \left[ \sum_{k=0}^\infty \E_{\mathbf{x}} \left( \int_{T_{k}}^{T_{k+1}} e^{-qt} r^{\pi(k)}\left(X^{\pi(k)}(\sigma^\pi_{\pi(k)}(k)+(t-T_k))\right) \diff t \Big| \mathcal{G}(T_k)\right) \right] \\
&= \E_{\mathbf{x}} \left[ \sum_{k=0}^\infty e^{-qT_k}\E_{\mathbf{x}} \left( \int_{0}^{\textrm{e}_\lambda} e^{-qt} r^{\pi(k)}\left(X^{(\pi(k))}(\sigma^\pi_{\pi(k)}(k)+t)\right) \diff t \Big| \mathcal{G}(T_k) \right) \right] \\
&= \E_{\mathbf{x}} \Big[ \sum_{k=0}^\infty e^{-q T_k} R^{\pi(k)}(X^{\pi(k)}(\sigma^\pi_{\pi(k)}(k))) \Big],
\end{split}
\end{align}
where
\begin{align} \label{connection_prob_0_1}
R^{j}(x) := 
\E_{x}^j \left[ \int_{0}^{\textrm{e}_\lambda} e^{-qt} r^{j}(X^{j}(t)) \diff t \right]
= \E_{x}^j \left[ \int_{0}^{\infty} e^{-(q+\lambda)t} r^{j}(X^{j}(t)) \diff t \right], \quad j \in \J \; \textrm{and} \; x \in \R.
\end{align}
Hence, this is a special case of Problem 1 with this choice of the reward function.

For Problem 2, we assume below, following what are assumed in  \cite{KM95}. 
\begin{assump} \label{assump_prob1}
For each $j \in \J$, we assume
\begin{enumerate}
\item the reward function $r^j$ is strictly increasing and continuous.
\item the integrability condition \eqref{integrability_r} is satisfied. 
\end{enumerate}
\end{assump}
This assumption guarantees that \eqref{connection_prob_0_1} is finite and Assumption \ref{assump_R_increasing} is fulfilled.

\subsection{Gittins index}

We define the Gittins index for our problems, analogously to those in the classical case.
For simplicity, we fix an arm $j \in \J$ and omit the index in an obvious manner (e.g.\ $R \equiv R^{j}$ and $X \equiv X^{j}$). Here, we follow similar steps as \cite[Section 3.6]{Mandelbaum88} to derive our Gittins index.

Recall that the Gittins index in the discrete-time case \eqref{gittins_index_discrete_time} is derived by considering the optimal stopping problem \eqref{opt_stopping_discrete} parameterized by $\gamma$.
Thus, it is natural to consider an analogous index in the Poissonian decision time settings, by considering \eqref{opt_stopping_discrete} with the deterministic decision times replaced by Poisson arrival times.  


 For each $\gamma \in \R$, consider the following auxiliary optimal stopping problem
\begin{align*}
v(x,\gamma)=\sup_{M\geq 1}\E_x \left[ \sum_{k=0}^{M-1} e^{-q T_k} \left(R(X(T_k)) -\gamma\right)\right],
\end{align*}
where $M$ is selected from a set of stopping times greater than or equal to $1$ with respect to the filtration $\mathbb{A}:=\left(\mathcal{A}_k\right)_{k\geq0}$ with $\mathcal{A}_k:=\sigma\left(T_k, 1 \leq i \leq k ,X(t), t \leq T_k \right)$, i.e. we stop at $T_M$ for an $\mathbb{A}$-stopping time $M$.


\begin{lemma} \label{lemma_convexity}
For $x\in\R$, the mapping $\gamma\mapsto v(x,\gamma)$ is strictly decreasing and convex. 
\end{lemma}
\begin{proof}

For $\overline{\gamma} > \underline{\gamma}$ and an $\mathbb{A}$-stopping time $M \geq 1$,
\begin{align*}
	\E_x \Big[ \sum_{k=0}^{M-1} e^{-q T_k} \left(R(X(T_k)) -\underline{\gamma}\right)\Big] = (\overline{\gamma} - \underline{\gamma}) + \E_x \Big[ \sum_{k=0}^{M-1} e^{-q T_k} \left(R(X(T_k)) - \overline{\gamma}\right) \Big].
\end{align*}	
After maximizing both sides over $M$, the decreasing property is immediate.


Regarding the convexity, for $\delta \in[0,1]$, $\gamma_1,\gamma_2\in\R$, and  $\mathbb{A}$-stopping time $M$,
\begin{multline*}
v_M(x,\delta \gamma_1+(1-\delta)\gamma_2) := \E_x \Big[ \sum_{k=0}^{M-1} e^{-q T_k} \left[R(X(T_k)) -\left(\delta \gamma_1+(1-\delta )\gamma_2\right)\right]\Big] \\
=\delta\E_x \Big[ \sum_{k=0}^{M-1} e^{-q T_k} \left[R(X(T_k)) -\gamma_1\right]\Big]+(1-\delta)\E_x \Big[ \sum_{k=0}^{M-1} e^{-q T_k} \left[R(X(T_k)) -\gamma_2\right]\Big]
\end{multline*}
where the decomposition makes sense from Assumption \ref{assump_R_increasing}(2) together with the fact that
$\E_x \left[ \sum_{k=0}^\infty e^{-q T_k}\right]=\E_x [ 1+\int_{(0,\infty)}e^{-qs}\diff N(s) ]=(\lambda+q)/q <\infty$.	
This gives 
$v_M(x,\delta \gamma_1+(1-\delta)\gamma_2) \leq \delta v(x,\gamma_1)+(1-\delta)v(x,\gamma_2)$ and hence
 $v(x,\delta \gamma_1+(1-\delta)\gamma_2)\leq \delta v(x,\gamma_1)+(1-\delta)v(x,\gamma_2)$, showing the convexity.
	
\end{proof}

Fix $x \in \R$. By Lemma \ref{lemma_convexity} we have that the mapping $\gamma\mapsto v(x,\gamma)$ is continuous, and monotone convergence gives $\lim_{\gamma \downarrow -\infty}v(x,\gamma)=\infty$ and $\lim_{\gamma \uparrow \infty} v(x,\gamma)=-\infty$. 
Hence, $v(x,\cdot)$ has a unique root, which we denote by $\Gamma(x)$ such that
\begin{align} \label{v_gamma_relation}
\begin{split}
v(x,\gamma) > 0 &\Longleftrightarrow \gamma < \Gamma(x), \\
v(x,\gamma) = 0 &\Longleftrightarrow \gamma = \Gamma(x), \\
v(x,\gamma) < 0 &\Longleftrightarrow \gamma > \Gamma(x).
\end{split}
\end{align}
We call $\Gamma(x)$ the Gittins index for our problem.

By these obtained characteristics of the optimal stopping problem, following the same arguments in Section 3.6 of \cite{Mandelbaum88}, we derive our Gittins index in the form analogous to \eqref{gittins_index_discrete_time}. 

By Snell's optimal-stopping theory (see, e.g., \cite{Snell}) we have that the stopping rule that attains $v(x,\gamma)$ is given by
\begin{align}\label{stopping_rule}
\begin{split}
\tau(x,\gamma):&=\inf\{m\geq 1: v(X(T_m),\gamma)\leq 0\} \\
&=\inf\{m\geq 1:\Gamma(X(T_m))\leq \gamma \}, 
\end{split}
\end{align}
where the latter holds by 
\eqref{v_gamma_relation}.



Therefore, for any $\mathbb{A}$-stopping time $M' \geq 1$,
\begin{align} \label{about_v_M}
\begin{split}
0=v(x,\Gamma(x))&=\sup_{M\geq 1}\E_x \Big[ \sum_{k=0}^{M-1} e^{-q T_k} \left[R(X(T_k)) -\Gamma(x)\right]\Big]\geq \E_x \Big[ \sum_{k=0}^{M'-1} e^{-q T_k} \left[R(X(T_k)) -\Gamma(x)\right]\Big]\\
&= \E_x \Big[ \sum_{k=0}^{M'-1} e^{-q T_k} R(X(T_k))\Big]-\Gamma(x)\E_x \Big[ \sum_{k=0}^{M'-1} e^{-q T_k} \Big],
\end{split}
\end{align}
and thus
\begin{align*}
\Gamma(x)\geq \frac{\E_x \Big[ \sum_{k=0}^{M'-1} e^{-q T_k} R(X(T_k))\Big]}{\E_x \Big[ \sum_{k=0}^{M'-1} e^{-q T_k}\Big]}.
\end{align*}
On the other hand, by \eqref{stopping_rule},  the inequality \eqref{about_v_M} holds with equality when $M'$ equals
\[
\tau(x,\Gamma(x))=\inf\{m\geq 1:\Gamma(X(T_m))\leq \Gamma(x) \}.
\]
 This implies that
\begin{align}\label{gittins_temp}
\Gamma(x)=\sup_{M \geq 1} \frac{\E_x \Big[ \sum_{k=0}^{M-1} e^{-q T_k} R(X(T_k))\Big]}{\E_x \Big[ \sum_{k=0}^{M-1} e^{-q T_k} \Big]}= \frac{\E_x \Big[ \sum_{k=0}^{\tau(x,\Gamma(x))-1 }e^{-q T_k} R(X(T_k))\Big]}{\E_x \Big[ \sum_{k=0}^{\tau(x,\Gamma(x)) -1} e^{-q T_k} \Big]}.
\end{align}
Note that both the numerator and denominator are finite by Assumption \ref{assump_R_increasing}(2).

We now write \eqref{gittins_temp} more explicitly. By Assumption  \ref{assump_R_increasing} and because $X$ is spatially homogeneous, $\Gamma(x)$ is strictly increasing in $x$ and thus 
\begin{equation}\label{tau}
\tau(x,\Gamma(x))=\inf\{m\geq 1:X(T_m) \leq x \}.
\end{equation}
Because $N$ is a Poisson process with arrival times $\mathcal{T} := (T_i)_{i\in\mathbb{N}}$, using
\[
T_x^- := T_{\tau(x,\Gamma(x))} = \inf \{ S \in \mathcal{T}: 
X(S) \leq x\},
\]
we have
\begin{align*}
\E_x \left[ \sum_{k=0}^{\tau(x,\Gamma(x))-1} e^{-q T_k} R(X(T_k))\right]&= R(x)+\E_x \left[ \int_{(0,T_x^-]} e^{-q s} R(X(s)) \diff N(s)\right] - \E_x \left[ e^{-q T_x^-} R(X(T_x^-))\right]\notag\\
&= R(x)+\E_x \left[\int_{(0,T_x^-]} e^{-q s} R(X(s)) 1_{\{ X(s) > x \}} \diff   N(s)\right] \notag\\
&=R(x)+\lambda\E_x \left[ \int_0^{T_x^-} e^{-q s} R(X(s)) 1_{\{ X(s) > x \}}  \diff  s\right].
\end{align*}
Here, the second equality holds because if $T_i < T_x^-$ then $X(T_i) > x$ for all $i \geq 1$ whereas $X(T_x^-) \leq x$ by the definition of $T_x^-$. The last equality holds by compensation formula.

We hence have the following.

\begin{theorem} 
The Gittins index $\Gamma$ for Problem 1 is 
\begin{align}
	\Gamma(x)= \frac{\displaystyle R(x)+\lambda\E_x\left[\int_0^{T_x^-}e^{-qs}R(X(s))1_{\{X(s)>x\}}\diff s \right]}{\displaystyle 1+\lambda\E_x\left[\int_0^{T_x^-}e^{-qs}1_{\{X(s)>x\}}\diff s \right]}, \quad x \in \R. \label{gittins}
\end{align}
\end{theorem}

The Gittins index for Problem 2 can be written by substituting \eqref{connection_prob_0_1} in \eqref{gittins}.




\begin{remark} \label{remark_lambda0}
As $\lambda \to 0$,  the Gittins index $\Gamma(x)$ converges to the instantaneous reward $R(x)$ for all $x \in \R$.
In particular, it converges to the $q$-resolvent 
$\E_{x} \left[ \int_{0}^{\infty} e^{-qt} r(X(t)) \diff t \right]$
for Problem 2. These results are consistent because, as $\lambda$ goes to zero, the future rewards vanish. 
\end{remark}



\section{Spectrally one-sided cases} \label{section_SN}

In this section, we focus on the spectrally negative and positive cases and compute the Gittins index \eqref{gittins} in terms of the scale function. Recall that a spectrally negative \lev process is  a L\'evy process without positive jumps that does not have monotone paths a.s. The spectrally positive \lev process is its dual process. 

\subsection{Review of fluctuation theory for spectrally negative \lev processes} \label{section_fluctuation_identities}
 We first review the fluctuation theory and scale function of the spectrally negative \lev process.

Let $(Y(t))_{t \geq 0}$ be a spectrally negative \lev process defined on a probability space $(\Omega, \mathcal{F}, \p)$.
We denote its Laplace exponent by
\begin{align} \label{laplace_exponent}
\psi(\theta) := \log \e\big[{\rm e}^{\theta Y(1)}\big], \qquad \theta\ge 0,
\end{align}
which is known to be convex on $[0, \infty)$ and admits its right-inverse
		\begin{align}
				\Phi(q) := \sup \{ s \geq 0: \psi(s) = q\}, \quad q \geq 0.
			\label{def_varphi}
		\end{align}
For $y \in \mathbb{R}$, let $\p_y$ be the conditional probability under which $Y(0) = y$ and $\E_y$ be its expectation operator. We omit the subscript when $y = 0$. 


Fix $q \geq 0$. The $q$-scale function $W^{({q})}$ is a mapping from $\R$ to $[0, \infty)$ that takes value zero on $(-\infty,0)$, while on $[0,\infty)$ it is a continuous and strictly increasing function with the Laplace transform
		\begin{align} \label{new_4_a}
			\begin{split}
				\int_0^\infty  \mathrm{e}^{-\theta x} W^{({q})}(x) \diff x &= \frac 1 {\psi(\theta)-q}, \quad \theta > \Phi({q}).
			\end{split}
		\end{align}
Define also the \emph{second scale function} 
\begin{align*} 
Z^{({q})}(x; \theta ) &:=e^{\theta x} \Big( 1 + (q- \psi(\theta )) \int_0^{x} e^{-\theta  z} W^{(q)}(z) \diff z	\Big), \quad x \in \R, \, \theta  \geq 0.
\end{align*}
In particular, for $x \in \R$ and $\lambda > 0$,
\begin{align*} 
	\begin{split}
		Z^{(q)}(x; \Phi(q+\lambda)) &=e^{\Phi(q+\lambda) x} \Big( 1 -\lambda \int_0^{x} e^{-\Phi(q+\lambda) z} W^{(q)}(z) \diff z \Big),
	\end{split}
\end{align*}
which can also be written by \eqref{new_4_a},
\begin{align} \label{Z_special_2}
	\begin{split}
		Z^{(q)}(x; \Phi(q+\lambda)) &=e^{\Phi(q+\lambda) x} \lambda \int_x^{\infty} e^{-\Phi(q+\lambda) z} W^{(q)}(z) \diff z = \lambda \int_0^{\infty} e^{-\Phi(q+\lambda) z} W^{(q)}(z+x) \diff z.
	\end{split}
\end{align}

There are a number of applications of the scale function.  For example, the $q$-resolvent can be written
\begin{align}\label{resol_sn}
\E_y\left[\int_0^{\infty}e^{-qs} 1_A(Y(s))\diff s \right]=\int_{\R} 1_A(y+u) g^{(q)}(u)\diff u,
\end{align}
for any Borel set $A$ on $\R$ where
\begin{equation}\label{resolvent_class}
g^{(q)}(u):=\frac{e^{-\Phi(q)u}}{\psi'(\Phi(q))}-W^{(q)}(-u),\qquad u\in\R;
\end{equation}
see \cite[Theorem 2.7 (iv)]{KKR}.



The joint Laplace transform of the first passage time
\begin{equation}\label{tauzero}\tau_0^-:=
\inf\{t > 0: Y(t)<0\} \end{equation}
and the overshoot $Y(\tau_0^-)$
is given by the following identity
\begin{align}\label{der_gamma_0_0}
\begin{split}
H^{(q)}(y; \theta) &:= \E_{y} \Big[ e^{-q \tau_0^- + \theta Y(\tau_0^-)} \mathbf{1}_{\{\tau_0^-< \infty \}}\Big]
=Z^{({q})}(y;\theta)-\frac{\psi(\theta)-q}{\theta-\Phi(q)}W^{({q})}(y), \quad y\in\R, \theta\geq 0;
\end{split}
\end{align}
see, e.g.,  Eqn.\ (4.5) in \cite{Kyprianou}.
In particular,
\[
H^{(q)}(y;\Phi(q+\lambda))
=Z^{(q)}(y;\Phi(q+\lambda))-\frac{\lambda}{\Phi(q+\lambda)-\Phi(q)}W^{(q)}(y), \quad y \in \R.
\]

Similar results have been obtained for the Poisson observation case.
Recall that $\mathcal{T} := (T_n)_{n \geq 1}$ is
the set of jump times of an independent Poisson process with rate $\lambda$. We define for $z \in \R$
\begin{align*}
\tilde{T}_{z}^-  := \inf \left\{ S \in \mathcal{T} : Y(S) < z \right\} \quad \textrm{and} \quad 
\tilde{T}_{z}^+  := \inf \left\{ S \in \mathcal{T} : Y(S) > z \right\}.
\end{align*}

Using  \cite[Theorem B.1]{PPSY} and \cite[Theorem 3.1]{L}, we have the following.

\begin{lemma} \label{thm_resolvent} 
Let $A$ be a Borel set on $\R$.

(1) We have
	\begin{align*}
		\E \left[ \int_0^{\tilde{T}_0^-  } e^{-qt} 1_A(Y(t))  \diff t \right]= \frac{\Phi(q+\lambda)-\Phi(q)}{\lambda} \int_{A}
		 H^{(q+\lambda)}(-u;\Phi(q)) \diff u.
	\end{align*}
%
%
%
%
%
(2) We have
\begin{align*}
\E\left[\int_0^{\tilde{T}_0^+}e^{-qt}1_A(-Y(t))\diff t\right] =\frac{\Phi(q+\lambda)-\Phi(q)}{\lambda}
\int_{A}H^{(q)}(u;\Phi(q+\lambda)) \diff u.
\end{align*}
\end{lemma}

%

\begin{thebibliography}{99} 




\bibitem{Albrecher} Albrecher, H., Ivanovs, J., Zhou, X. Exit identities for L\'evy processes observed at Poisson arrival times.
{\it Bernoulli}, \textbf{22},  1364--1382, 2016.


	\bibitem{Avanzi_Cheung_Wong_Woo} Avanzi, B., Cheung, E. C., Wong, B., Woo, J.K. On a periodic dividend barrier strategy in the dual model with continuous monitoring of solvency.
{\it Insurance Math. Econom.}, \textbf{52}, 98--113, 2013. 





\bibitem{Be}  Bertoin, J. {\it L\'evy processes. }\rm Cambridge University Press, Cambridge, 1996.

\bibitem{Carr} Carr, P.  Randomization and the American put. {\it Rev. Financ. Stud.}, \textbf{11} (3), 597-626, 1998.

\bibitem{Czarna} Czarna, I., Palmowski, Z. Dividend problem with Parisian delay for a spectrally negative L\'evy risk process. {\it J. Optim. Theory Appl.}, {\bf 161(1)}, 239-256, 2014.

\bibitem{Gittins} Gittins, J.  A dynamic allocation index for the sequential design of experiments. {\it Progress in statistics}, 241-266, 1974.

\bibitem{ElKaroui_Karatzas}El Karoui, N., \& Karatzas, I.  Dynamic allocation problems in continuous time. \textit{Ann. Appl. Probab.}, 255-286, 1994.

\bibitem{KM95} Kaspi, H., Mandelbaum, A. Multi armed bandits driven by L\'evy processes. \textit{Ann. Appl. Probab.}  \textbf{5} 2, 541--565, 1995.


\bibitem{Mandelbaum88} Kaspi, H., Mandelbaum, A. Multi armed bandits in discrete and continuous time. \textit{Ann. Appl. Probab.}  \textbf{8} 4, 1270--1290, 1998.


\bibitem{KKR}  Kuznetsov, A., Kyprianou, A.E., \& Rivero, V. \rm The
theory of scale functions for spectrally negative L\'evy processes. {\it L\'evy Matters II, Springer Lecture Notes in Mathematics}, 2013.
\bibitem{K}  Kyprianou, A.E. {\it Fluctuations of L\'evy processes with applications.} \rm Springer,
Berlin, 2014.



\bibitem{Kyprianou} Kyprianou, A.E., Surya, B.A.:
Principles of smooth and continuous fit in the determination of endogenous bankruptcy levels.
{\it Finance Stoch.}, \textbf{11}, 131--152, 2007.


\bibitem{L} Landriault, D., Li, B., Wong, J. T., Xu D., Poissonian potential measures for L\'evy risk models,
\textit{Insurance Math. Econom.}, \textbf{82}, 152--166, 2018.

\bibitem{Leung} Leung, T., Yamazaki, K., \& Zhang, H. An analytic recursive method for optimal multiple stopping: Canadization and phase-type fitting. {\it  Int. J. Theor. Appl. Finance}, \textbf{18} (05), 1550032, 2015.


\bibitem{Lkabous}
Lkabous, M. A. Poissonian occupation times of spectrally negative L\'evy processes with applications. {\it Scandinavian Actuarial Journal}, {\bf 2021(10)}, 916-935, 2021.


\bibitem{MPY} Mata, D., P\'erez, J.L.\ and Yamazaki, K. \rm Effects of Positive Jumps of Assets on Endogenous Bankruptcy and Optimal
Capital Structure: Continuous- and Periodic-Observation Models.  {\it SIAM J Financial Math.}, \textbf{12}(3) 1112--1149, 2021.

\bibitem{Noba} Noba, K.  Generalized scale functions of standard processes with no positive jumps. {\it Electron. Commun. Probab.}, \textbf{25}, 1-12, 2020.
	
			\bibitem{PPY} Palmowski, Z., P\'erez, J. L., Yamazaki, K.  Double continuation regions for American options under Poisson exercise opportunities. {\it Math. Financ.}, {\bf 31}(2), 722-771. 2021.



\bibitem{PPSY} Palmowski, Z., P\'erez, J. L., Surya, B. A., Yamazaki, K. The Leland-Toft optimal capital structure
model under Poisson observations, \textit{Finance Stoch.}, \textbf{24}, 1035--1082, 2020.

\bibitem{PY} P\'erez, J.L.\ and Yamazaki, K. \rm Mixed periodic-classical barrier strategies for L\'evy risk processes.  {\it Risks}, \textrm{6}(2), 33, 2018.


	\bibitem{Perez_Yamazaki_Bensoussan}  P\'erez, J.L., Yamazaki, K., Bensoussan, A.  Optimal periodic replenishment policies for spectrally positive L\'evy demand processes. {\it SIAM J. Control Optim.}, {\bf 58}(6), 3428--3456, 2020.

\bibitem{Puterman}
Puterman, M. L.  Markov decision processes: discrete stochastic dynamic programming. John Wiley \& Sons, 2014.


\bibitem{Snell} Snell, L. \textit{Applications of martingale systems theorems.}  \textit{Trans. Amer. Math. Soc.}  \textbf{73}, 293-312, 1952.



		\bibitem{ZCY}
Zhao, Y., Chen, P., Yang, H. Optimal periodic dividend and capital injection problem for spectrally positive Lévy processes. {\it Insurance Math. Econom.}, {\bf 74}, 135-146, 2017.

\end{thebibliography}
\subsection{Gittins index for spectrally one-sided \lev processes} \label{section_gittins_SN}
We now express the Gittins index \eqref{gittins} in terms of the scale function for both spectrally negative and positive cases.





\begin{proposition} \label{prop_SN}
Suppose $X$ is a  spectrally negative \lev process with its Laplace exponent $\psi$ and scale function $W^{(q)}$ as defined in Section \ref{section_fluctuation_identities}.

\begin{enumerate}
\item The Gittins index \eqref{gittins} for Problem 1 can be written, for $x \in \R$,
\begin{align} \label{Gittins_SN_1_a}
\begin{split}
\Gamma(x) 
&= \frac {\Phi(q)} {\Phi(q+\lambda)} 
 \left( R(x)+ (\Phi(q+\lambda) - \Phi(q)) \int_0^\infty  R(x+y)e^{-\Phi(q) y} \diff y \right).
\end{split}
\end{align}
\item  The Gittins index \eqref{gittins} for Problem 2 can be written, for $x \in \R$,
\begin{multline}\label{gitt_prob_1_sn}
\Gamma(x)
=\frac{ \Phi(q)}{\Phi(q+\lambda)}\Bigg[\int_{\R}r(x+u)g^{(q+\lambda)}(u)\diff u \\+(\Phi(q+\lambda)-\Phi(q))\int_0^{\infty}  e^{-\Phi(q)y}\int_{\R} r(x+y+u)g^{(q+\lambda)}(u)\diff u\diff y\Bigg], 
\end{multline}
where $g^{(q+\lambda)}$ is defined as in \eqref{resolvent_class} (with $q$ replaced by $q+\lambda$).
\end{enumerate}
\end{proposition}
\begin{proof}

(1) By Lemma \ref{thm_resolvent}(1) with $Y=X$ and $h(\cdot) = r(x+\cdot )$ and because $H^{(q+\lambda)}(-y;\Phi(q)) =
e^{-\Phi(q) y}$ for all $y \geq 0$,
\begin{align*}
\E_x\left[\int_0^{T_x^-}e^{-qs}R(X(s))1_{\{X(s)>x\}}\diff s \right]&=\E\left[\int_0^{T_0^-}e^{-qs}R(X(s) +x)1_{\{X(s)>0\}}\diff s \right] \\ &=\frac{\Phi(q+\lambda)-\Phi(q)}{\lambda}\int_0^\infty R(x+y)
e^{-\Phi(q) y}
\diff y.
\end{align*}
By replacing $R(\cdot)$ with $1$, we also have $\E_x\left[\int_0^{T_x^-}e^{-qs}1_{\{X(s)>x\}}\diff s \right] = (\Phi(q+\lambda)-\Phi(q))/(\lambda \Phi(q))$. Substituting these in \eqref{gittins}, we have  \eqref{Gittins_SN_1_a}.

(2) 
By \eqref{connection_prob_0_1} and \eqref{resol_sn}, we have
$R(x)
=\int_{\R}r(x+u)g^{(q+\lambda)}(u)\diff u$, which is finite by Assumption \ref{assump_prob1}.
Substiuting this in \eqref{Gittins_SN_1_a}, we have \eqref{gitt_prob_1_sn}.

\end{proof}

\begin{proposition} \label{prop_SP}
Suppose $X$ is a  spectrally positive \lev process whose dual process $-X$ is a spectrally negative \lev process with its Laplace exponent $\psi$ and scale function $W^{(q)}$. 
\begin{enumerate}
\item The Gittins index for Problem 1 is given by, for $x \in \R$,
\begin{align}\label{Gittins_SP_1}
\Gamma(x)= \frac{q \Phi(q+\lambda)}{(\lambda+q)\Phi(q)} \left(R(x)+ (\Phi(q+\lambda)-\Phi(q)) \int_0^\infty R(x+y)H^{(q)}(y;\Phi(q+\lambda))\diff y \right). 
\end{align}
\item The Gittins index for Problem 2 is given by, for $x \in \R$,
\begin{multline}\label{gitt_prob_1_sp}
\Gamma(x)
=\frac{q \Phi(q+\lambda)}{(\lambda+q)\Phi(q)}\left(\int_{\R}r(x-u)g^{(q+\lambda)}(u)\diff u \right. \\ \left. +(\Phi(q+\lambda)-\Phi(q)) \int_0^\infty  H^{(q)}(y;\Phi(q+\lambda))  \int_{\R} r(x+y-u)g^{(q+\lambda)}(u)\diff u \diff y \right). 
\end{multline}

\end{enumerate}
\end{proposition}
\begin{proof}
(1) By Lemma \ref{thm_resolvent}(2) with $Y = -X$, 
\begin{align*}
\E_x\left[\int_0^{T_x^-}e^{-qs}R(X(s))1_{\{X(s)>x\}}\diff s \right]
 &=\E\left[\int_0^{\tilde{T}_0^+}e^{-qs}R(-Y(s)+x)1_{\{-Y(s)>0\}}\diff s \right] \\
&= \frac{\Phi(q+\lambda)-\Phi(q)}{\lambda} \int_0^\infty R(x+y)H^{(q)}(y;\Phi(q+\lambda))\diff y,
\end{align*}
which is finite by Assumption \ref{assump_R_increasing}(2).
On the other hand, 
\begin{align}\label{new_6_a}
	\E_x\left[\int_0^{T_x^-}e^{-qs}1_{\{X(s)>x\}}\diff s \right]&=\E_x\left[\int_0^{T_x^-}e^{-qs}\diff s \right]-\E_x\left[\int_0^{T_x^-}e^{-qs}1_{\{X(s)\leq x\}}\diff s \right].
\end{align}
Here, by identity (62) in \cite{MPY}, 
\begin{align*} 
	\E_x\left[\int_0^{T_x^-}e^{-qs}\diff s \right] = \frac{1}{q}\left(1-\E_x\left[e^{-qT_x^-}\right]\right)=\frac{1}{q}\left(1-\frac{\Phi(q+\lambda)-\Phi(q)}{\Phi(q+\lambda)}\right)=\frac{1}{q}\frac{\Phi(q)}{\Phi(q+\lambda)},
\end{align*}
and, by Lemma \ref{thm_resolvent}(2) with $Y = -X$ and $A = (-\infty,0]$, 
\begin{multline*} 
\E_x\left[\int_0^{T_x^-}e^{-qs}1_{\{X(s)\leq x\}}\diff s \right] = \E\left[\int_0^{\tilde{T}_0^+}e^{-qs}1_{\{-Y(s)\leq  0\}}\diff s \right] = \int_{-\infty}^0 \frac{\Phi(q+\lambda)-\Phi(q)}{\lambda}H^{(q)}(y,\Phi(q+\lambda))\diff y  \\= \int_0^{\infty} \frac{\Phi(q+\lambda)-\Phi(q)}{\lambda}e^{-\Phi(q+\lambda)y}\diff y=\frac{\Phi(q+\lambda)-\Phi(q)}{\Phi(q+\lambda)\lambda}.
\end{multline*}
Substituting these in 
\eqref{new_6_a} gives
\begin{align*}
1 + \lambda \E_x\left[\int_0^{T_x^-}e^{-qs}1_{\{X(s)>x\}}\diff s \right]=1 + \lambda 
\left( \frac{1}{q}\frac{\Phi(q)}{\Phi(q+\lambda)} - \frac{\Phi(q+\lambda)-\Phi(q)}{\Phi(q+\lambda)\lambda} \right)
= \frac{(\lambda+q)\Phi(q)}{q \Phi(q+\lambda)}.
\end{align*}
Substituting these we have \eqref{Gittins_SP_1}.

(2) 
By \eqref{connection_prob_0_1} and \eqref{resol_sn},
\begin{align*}
R(x)=\E\left[\int_0^{\infty}e^{-(q+\lambda)s}r(X(s)+x)\diff s\right]
=\E\left[\int_0^{\infty}e^{-(q+\lambda)s}r(-Y(s)+x)\diff s\right] 
=\int_{\R}r(x-u)g^{(q+\lambda)}(u)\diff u,
\end{align*}
which is finite by Assumption \ref{assump_prob1}(2).
Substiuting this in \eqref{Gittins_SP_1}, we have \eqref{gitt_prob_1_sp}.
%
%
%
%
\end{proof}

%
%
%

\section{Convergence of the Gittins index as $\lambda \to \infty$} \label{section_convergence}

In this section, we show the convergence of the Gittins indices obtained in Propositions \ref{prop_SN} and \ref{prop_SP} to those in the classical \lev case  \cite{Mandelbaum88}.

\subsection{Classical \lev bandit case} \label{section_classical_case_SN}

We first review and obtain a characterization of the Gittins index of \cite{Mandelbaum88}.
Let $Y$ be a spectrally negative \lev process and use the same notations used in Section \ref{section_fluctuation_identities}.

For the spectrally negative case, the local time at maximum can be selected to be $L(t) =\overline{Y}(t) := \sup_{0 \leq u \leq t} {Y}(u)$ for $t \geq 0$. Then, the ascending ladder height process becomes $(L^{-1}(t), H(t))_{t \geq 0}$ where
%
\[
L^{-1}(t) := \inf \{s > 0: \overline{Y}(s) > t \} =  \inf \{s > 0: Y(s) > t \},
\]
and because $Y$ does not jump upwards,
\[
H(t) := Y(L^{-1}(t)) = t,
\]
(see \cite[Section 6.5.2]{K}). As in \cite[(6.34)]{K}, the pair $(L^{-1}(t), H(t))_{t \geq 0}$ becomes a two-dimensional subordinator with its Laplace exponent 
\begin{align} \label{laplace_exp_ladder_SN}
\E [e^{- q L^{-1}(1) - \theta H(1)} 1_{\{ 1 < L(\infty)\}}] = e^{- \kappa(q, \theta)}
\end{align}
where
\[
\kappa(q, \theta) := \Phi(q)+\theta, \quad q, \theta \geq 0.
\]
By the Wiener-Hopf factorization, with $\widehat{L}$ and $(\widehat{L}^{-1}, \widehat{H})$ those for the dual process $-Y$ (called the descending ladder height process)
\begin{align}
\E [e^{- q \widehat{L}^{-1}(1) - \theta \widehat{H}(1)} 1_{\{ 1 < \widehat{L}(\infty)\}}] = e^{- \widehat{\kappa}(q, \theta)}  \label{laplace_exp_ladder_SP}
\end{align}
with
\[\widehat{\kappa}(q, \theta) := \frac {q - \psi(\theta)} {\Phi(q) - \theta}, \quad q, \theta \geq 0;
\]
see \cite[Eqn.\ (6.35)]{K}. 

%


Using these, the Laplace transform \eqref{mu_laplace2} that characterizes the Gittins index in the classical \lev model can be written explicitly.

(1) Suppose $X$ is a spectrally negative \lev process. By \eqref{laplace_exp_ladder_SN} for $Y=X$,
\[
\bar{\varphi} (q, \theta) = \kappa(q, \theta) = \Phi(q)+\theta, \quad q, \theta \geq 0,
\]
and thus \eqref{mu_laplace2} becomes
\begin{align}  \label{mu_laplace_classical_SN}
\int_{[0,\infty)} e^{-\theta y} \mu (\diff y) = \frac {\Phi(q)} {\Phi(q)+\theta}, \quad \theta \geq 0.
\end{align}

(2) Suppose $X$ is a spectrally positive \lev process. Let $\psi$ be the Laplace exponent of its dual (spectrally negative) \lev process and $\Phi$ its right inverse. By \eqref{laplace_exp_ladder_SP} for $Y=-X$,
\[
\bar{\varphi} (q, \theta) =\widehat{\kappa}(q, \theta) = \frac {q - \psi(\theta)} {\Phi(q) - \theta}, \quad  \theta \geq 0,
\]
and \eqref{mu_laplace2} becomes
\begin{align}  \label{mu_laplace_classical_SP}
\int_{[0,\infty)} e^{-\theta y} \mu (\diff y) = \frac{q}{\Phi(q)}\frac{\theta-\Phi(q)}{\psi(\theta)-q}, \quad  \theta \geq 0.
\end{align}

\subsection{Convergence of the Gittins index}
Suppose $X$ is  spectrally negative. For Problem 1, the Gittins index \eqref{Gittins_SN_1_a} can be written
\begin{align} \label{mu_1_lambda}
 \Gamma(x) 
	&=\int_{[0,\infty)}R(x+y)\mu_1^{\lambda}(\diff y), \quad x \in \R,
\end{align}
with 
\begin{align*}
	\mu_1^{\lambda}(\diff y)&:=\frac{\Phi(q)}{\displaystyle \Phi(q+\lambda) } \left( \delta_0(\diff y)+ 1_{\{y>0\}} (\Phi(q+\lambda)-\Phi(q))e^{-\Phi(q)y}\diff y\right), \quad y \geq 0,
\end{align*}
where $\delta_0$ is the Dirac measure at zero.

For Problem 2, the (normalized) Gittins index \eqref{gitt_prob_1_sn} can be written
\begin{align}
\begin{split}
(q+\lambda)\Gamma(x) &=\frac{(q+\lambda)\Phi(q)}{\Phi(q+\lambda)}\left[\int_{\R}r(x+u)g^{(q+\lambda)}(u)\diff u \right. \\ &\qquad \qquad \left. +(\Phi(q+\lambda)-\Phi(q))\int_0^{\infty}\int_{\R} r(x+w)e^{-\Phi(q)y}g^{(q+\lambda)}(w-y)\diff w\diff y \right] \\
&=\int_{\R}r(x+u)\mu_2^{\lambda}(\diff u), \quad x \in \R, \label{mu_2_lambda_int}
\end{split}
\end{align}
where 
\begin{align*}
\mu_2^\lambda(\diff u):=\frac{(q+\lambda)\Phi(q)}{\displaystyle \Phi(q+\lambda) }\left[g^{(q+\lambda)}(u)+ (\Phi(q+\lambda)-\Phi(q)) (f^{(q)}\star g^{(q+\lambda)})(u)\right] \diff u, \quad u \in \R.
\end{align*}
Here,
\[
(f^{(q)}\star g^{(q+\lambda)})(u) := \int_0^{\infty}e^{-\Phi(q)y}g^{(q+\lambda)}(u-y)\diff y
\]
which is a convolution of 
$f^{(q)}(u):= e^{-\Phi(q)u}1_{\{u>0\}}$ and $g^{(q+\lambda)}$ as defined in \eqref{resolvent_class}. 

In \eqref{mu_2_lambda_int}, we consider a normalized version because $R(\cdot)$ in \eqref{connection_prob_0_1} depends on $\lambda$.  This normalization is appropriate in view of Remark \ref{remark_prob_measure}(2) below. 

\begin{remark} \label{remark_prob_measure}
\begin{enumerate}
\item By setting $R \equiv 1$ in \eqref{mu_1_lambda}, we have $1 = \Gamma(x) = \mu_1^{\lambda}([0,\infty))$ and hence $\mu_1^{\lambda}$ is a probability measure.
\item When $r \equiv 1$, $R = (q+\lambda)^{-1}$ (see \eqref{connection_prob_0_1}) and  thus $\Gamma(x) \equiv (q+\lambda)^{-1}$. Substituting these in  \eqref{mu_2_lambda_int}, we have $\mu_2^\lambda(\R) = 1$, and thus $\mu_2^{\lambda}$ is a probability measure.
\end{enumerate}
\end{remark}

\begin{proposition}\label{convergence} When $X$ is spectrally negative, the measures $\mu_1^\lambda$ and $\mu_2^\lambda$ converge weakly to $\mu$ (defined by \eqref{mu_laplace_classical_SN}) as $\lambda \to \infty$.
\end{proposition}

\begin{proof}

\underline{Proof for Problem 1}.
For $\theta \geq 0$, because  
\[
\int_{[0,\infty)} e^{-\theta y } [\delta_0(\diff y)+ (\Phi(q+\lambda)-\Phi(q))1_{\{y>0\}}e^{-\Phi(q)y}\diff y] = 1 + \frac{\Phi(q+\lambda)-\Phi(q)}{\Phi(q)+\theta} = \frac {\theta + \Phi(q+\lambda)} { \Phi(q)+\theta}, 
\]
the Laplace transform of $\mu_1^{\lambda}$ becomes
\begin{align}
	\int_{[0, \infty)}e^{-\theta y}\mu_1^{\lambda}(\diff y)=
	\frac { \Phi(q)} { \Phi(q+\lambda)}  \frac {\theta + \Phi(q+\lambda)} {\Phi(q)+\theta }, 
 \label{laplace_transform_mu}
\end{align}
which converges to \eqref{mu_laplace_classical_SN} as $\lambda \to \infty$.
By the the continuity theorem for Laplace transforms, the  weak convergence holds.

\underline{Proof for Problem 2}.
Because $\mu_2^\lambda$ has a support $\R$, we consider Fourier transforms. To this end, define  the characteristic exponent of $X =Y$ by $\Psi(\theta) = - \psi(\mathrm{i} \theta)$ such that 
\begin{equation}
\E [e^{\mathrm{i} \theta X(s)}] = e^{-\Psi(\theta)s}, \quad \theta \in \R, s \geq 0. \label{characteristic_X}
\end{equation}

Fix $\theta \in \R$.
Using \eqref{resol_sn} and \eqref{characteristic_X}, the Fourier transform of $g^{(q+\lambda)}$ becomes
\begin{align} \label{fourier_g}
\int_{\R}e^{\mathrm{i}\theta y}g^{(q+\lambda)}(y)\diff y=\E\left[\int_{0}^{\infty}e^{-(q+\lambda)s}e^{\mathrm{i}\theta X(s)}\diff s\right]&=\int_{0}^{\infty}e^{-(q+\lambda)s}e^{-\Psi(\theta)s}\diff s=\frac{1}{\Psi(\theta)+(q+\lambda)}.
\end{align}
On the other hand, the Fourier transform of $f^{(q)}$ becomes 
$\int_{\R}e^{\mathrm{i}\theta y}f^{(q)}(y) \diff y =
\int_{0}^{\infty}e^{\mathrm{i}\theta y}e^{-\Phi(q)y}\diff y=(\Phi(q)-\mathrm{i}\theta)^{-1}$.
By these and the convolution theorem,
\[
\int_{\R}e^{\mathrm{i}\theta y}(f^{(q)}\star g^{(q+\lambda)})(y) \diff y = \frac{1}{(\Psi(\theta)+(q+\lambda)) (\Phi(q)-\mathrm{i}\theta)}.
\]
This together with \eqref{fourier_g} gives

\begin{align}\label{laplace_transform_mu_2} 
\begin{split}
\int_{\R}e^{\mathrm{i}\theta y}\mu_2^{\lambda}(\diff y)&=\frac{ \Phi(q)}{\Phi(q+\lambda)}\frac{(q+\lambda)}{\Psi(\theta)+(q+\lambda)}\left(1+\frac{\Phi(q+\lambda)-\Phi(q)}{\Phi(q)-\mathrm{i}\theta}\right)\\
&=\frac{\Phi(q+\lambda)-\mathrm{i}\theta}{\Phi(q+\lambda)}\frac{(q+\lambda)}{\Psi(\theta)+(q+\lambda)}\frac{\Phi(q)}{\Phi(q)-\mathrm{i}\theta},
\end{split}
\end{align}
which converges  as $\lambda \to \infty$ to
$\int_{[0,\infty)}e^{\mathrm{i}\theta y}\mu(\diff y) = \Phi(q) / (\Phi(q)-\mathrm{i}\theta)$,
matching \eqref{mu_laplace_classical_SN} (with $\theta$ replaced by $- \mathrm{i} \theta$).
By the continuity theorem, the weak convergence holds.
\end{proof}
\begin{remark}
(1) For Problem 1, the Laplace transform \eqref{laplace_transform_mu} can be written
\begin{align*}
	\int_{[0, \infty)}e^{-\theta y}\mu_1^{\lambda}(\diff y)=\frac{\varphi_1(0,q,\lambda)}{\varphi_1(\theta,q,\lambda)}  \quad \textrm{where} \; \varphi_1(\theta,q,\lambda) :=  \frac {\Phi(q)+\theta} {\theta + \Phi(q+\lambda)}.
\end{align*}
(2) For Problem 2, the Fourier transform \eqref{laplace_transform_mu_2}  can be written
\begin{align*}
	\int_{\R}e^{ \mathrm{i} \theta y}\mu_2^{\lambda}(\diff y)=\frac{\varphi_2(0,q,\lambda)}{\varphi_2(\theta,q,\lambda)}  \quad \textrm{where} \; \varphi_2(\theta,q,\lambda) :=  \frac{(\Psi(\theta)+(q+\lambda)) (\Phi(q)-\mathrm{i}\theta)} {\Phi(q+\lambda)-\mathrm{i}\theta}.
\end{align*}
These expressions can be seen as  generalizations of the classical case \eqref{mu_laplace2}.
\end{remark}

We now consider the spectrally positive case. For Problem 1,
the Gittins index \eqref{Gittins_SP_1} can be written
\begin{align*}
	\Gamma(x)=\int_{[0,\infty)}R(x+y)\mu_1^{\lambda}(\diff y), \quad x \in \R,
\end{align*}
with 
\begin{align*}
	\mu_1^{\lambda}(\diff y)&:= \frac {q \Phi(q+\lambda)} {(\lambda+q)\Phi(q)} \left( \delta_0(\diff y)+ 1_{\{y>0\}} (\Phi(q+\lambda)-\Phi(q))H^{(q)}(y,\Phi(q+\lambda))\diff y\right), \quad y \geq 0.
\end{align*}
For Problem 2, the Gittins index \eqref{gitt_prob_1_sp} can be written
with $\widehat{g}(u) := g(-u)$, $u \in \R$,
\begin{align*}
(q + \lambda) \Gamma(x)
&=\frac{q \Phi(q+\lambda)}{\Phi(q)}\Bigg[\int_{\R}r(x-u)g^{(q+\lambda)}(u)\diff u \\ &+(\Phi(q+\lambda)-\Phi(q)) \int_0^\infty   \int_{\R}  H^{(q)}(y;\Phi(q+\lambda)) r(x+y-u)g^{(q+\lambda)}(u)\diff u \diff y\Bigg] \\
&=\frac{q \Phi(q+\lambda)}{\Phi(q)}\Bigg[\int_{\R}r(x+u)\widehat{g}^{(q+\lambda)}(u)\diff u \\ &+(\Phi(q+\lambda)-\Phi(q))\int_{\R}   \int_0^\infty    H^{(q)}(y;\Phi(q+\lambda)) r(x+w)\widehat{g}^{(q+\lambda)}(w-y) \diff y \diff w\Bigg] \\
&= \int_{\R}r(x+u)\mu_2^{\lambda}(\diff u)
\end{align*}
where
\begin{align*}
	\mu_2^\lambda(\diff u):=\frac {q \Phi(q+\lambda)} {\Phi(q)}\left(\widehat{g}^{(q+\lambda)}(u) + (\Phi(q+\lambda)-\Phi(q)) (f^{(q)}\star \widehat{g}^{(q+\lambda)})(u) \right) \diff u, \quad u \in \R,
\end{align*}
with
$f^{(q)}(u):=H^{(q)}(u,\Phi(q+\lambda)) 1_{\{u > 0 \}}$.

\begin{remark} Similar to Remark \ref{remark_prob_measure}, both  $\mu_1^{\lambda}$ and $\mu_2^{\lambda}$ are probability measures.
\end{remark}

\begin{proposition} \label{proposition_SP}
 When $X$ is spectrally positive, the measures $\mu_1^\lambda$ and $\mu_2^\lambda$ converge weakly to $\mu$ (defined by \eqref{mu_laplace_classical_SP}) as $\lambda \to \infty$.
\end{proposition}

\begin{proof}

\underline{Proof for Problem 1}.
In order to characterize the measure $\mu^{\lambda}$ we compute its Laplace transform. To this end, for $\theta > \Phi(q)$, using \eqref{Z_special_2}, 
\begin{align}\label{new_2_a}
	\int_0^{\infty}e^{-\theta y}Z^{(q)}(y;\Phi(q+\lambda))\diff y
	&=\lambda\int_0^{\infty}e^{-\theta y}\int_y^{\infty}e^{-\Phi(q+\lambda)(u-y)}W^{(q)}(u)\diff u \diff y\notag\\
	&=\lambda\int_0^{\infty}e^{-\Phi(q+\lambda)u}W^{(q)}(u)\int_0^{u}e^{-(\theta-\Phi(q+\lambda)) y}\diff y\diff u \notag\\
	&=\frac{\lambda}{\Phi(q+\lambda)-\theta}\int_0^{\infty}e^{-\Phi(q+\lambda)u}W^{(q)}(u)\left(e^{(\Phi(q+\lambda)-\theta)u}-1\right)\diff u \notag\\
	&=\frac{\lambda}{\Phi(q+\lambda)-\theta}\left(\int_0^{\infty}e^{-\theta u}W^{(q)}(u)\diff u-\int_0^{\infty}e^{-\Phi(q+\lambda)u}W^{(q)}(u)\diff u \right)\notag\\
	&=\frac{\lambda+q-\psi(\theta)}{(\Phi(q+\lambda)-\theta)(\psi(\theta)-q)},
\end{align}
and hence, again using \eqref{new_4_a}, 
\begin{align}
	(\Phi(q+\lambda)-\Phi(q)) \int_0^\infty e^{-\theta y}&H^{(q)}(y,\Phi(q+\lambda))\diff y
	\notag\\
	&=(\Phi(q+\lambda)-\Phi(q)) \frac{\lambda+q-\psi(\theta)}{(\Phi(q+\lambda)-\theta)(\psi(\theta)-q)}-\frac{\lambda}{\psi(\theta)-q}\notag\\
	&=\frac{(q-\psi(\theta))\Phi(q+\lambda)-(\lambda+q-\psi(\theta))\Phi(q)+\lambda\theta}{(\psi(\theta)-q)(\Phi(q+\lambda)-\theta)}, \label{new_5_a}
\end{align}
which holds for any $\theta \geq 0$  by analytic continuation and is finite in view of Lemma  \ref{thm_resolvent}(2).
Thus,
\begin{align*}
	&\int_{[0, \infty)} e^{-\theta y}\left[ \delta_0(\diff y)+ 1_{\{y>0\}} (\Phi(q+\lambda)-\Phi(q))H^{(q)}(y,\Phi(q+\lambda))\diff y\right] \\
	&= 1 + \frac{(q-\psi(\theta))\Phi(q+\lambda)-(\lambda+q-\psi(\theta))\Phi(q)+\lambda\theta}{(\psi(\theta)-q)(\Phi(q+\lambda)-\theta)} \\
&= \frac{(\lambda + q-\psi(\theta) ) (\theta - \Phi(q)) }{(\psi(\theta)-q)(\Phi(q+\lambda)-\theta)}.
\end{align*}
Therefore, 
\begin{align}
		\int_{[0,\infty)} e^{-\theta y}\mu_1^{\lambda}(\diff y)=
		\frac {\lambda + q-\psi(\theta) } {\lambda+q}  \frac {\Phi(q+\lambda)} {\Phi(q+\lambda)-\theta} \frac{ \theta - \Phi(q) }{ \psi(\theta)-q} {\displaystyle \frac{q}{\Phi(q)}}, \quad \theta \geq 0,
		\label{laplace_transform_mu_SP}
\end{align}
which converges to  \eqref{mu_laplace_classical_SP} as $\lambda \to \infty$.
By the continuity theorem, the weak convergence holds. 

\underline{Proof for Problem 2}.
Again let $\Psi(\theta) = - \psi(\mathrm{i} \theta)$ (where we recall $\psi$ is for the dual spectrally negative \lev process $Y = -X$) and hence $\E [e^{\mathrm{i} \theta Y(s)}] = e^{-\Psi(\theta)s}$, $\theta \in \R, s \geq 0$. 

We have
\begin{align} \label{fourier_g_dual}
\int_{\R}e^{\mathrm{i}\theta y}\widehat{g}^{(q+\lambda)}(y)\diff y=\E\left[\int_{0}^{\infty}e^{-(q+\lambda)s}e^{-\mathrm{i}\theta Y(s)}\diff s\right]&=\int_{0}^{\infty}e^{-(q+\lambda)s}e^{-\Psi(-\theta)s}\diff s=\frac{1}{\Psi(-\theta)+(q+\lambda)}.
\end{align}

%
Because \eqref{new_5_a} also holds for $\theta=0$ and it is finite, 
by analytic continuation, we obtain
\begin{align*}
(\Phi(q+\lambda)-\Phi(q)) \int_0^\infty e^{\mathrm{i}\theta y}&H^{(q)}(y,\Phi(q+\lambda))\diff y=\frac{(q+\Psi(-\theta))\Phi(q+\lambda)-(\lambda+q+\Psi(-\theta))\Phi(q)-i\lambda\theta}{(-\Psi(-\theta)-q)(\Phi(q+\lambda)+\mathrm{i}\theta)}.
\end{align*}

Together with \eqref{fourier_g_dual} and the convolution theory for Fourier transforms,
\begin{align}\label{laplace_transform_mu_SP_2}
\begin{split}
	\int_{\R}e^{\mathrm{i}\theta y}\mu_2^{\lambda}(\diff y)&=\frac {q \Phi(q+\lambda)} {\Phi(q)}\frac{1}{\Psi(-\theta)+(q+\lambda)}\left(1+\frac{(q+\Psi(-\theta))\Phi(q+\lambda)-(\lambda+q+\Psi(-\theta))\Phi(q)-i\lambda\theta}{(-\Psi(-\theta)-q)(\Phi(q+\lambda)+\mathrm{i}\theta)}\right) \\
	&=\frac { \mathrm{i}\theta + \Phi(q)    } {\Phi(q)}
	\frac {q} {\Psi(-\theta)+q}
	\frac{ \Phi(q+\lambda)}{\Phi(q+\lambda)+ \mathrm{i}\theta},
	\end{split}
\end{align}
which converges as $\lambda \to \infty$ to
\begin{align*}
\int_{[0,\infty)} e^{\mathrm{i} \theta y} \mu (\diff y) = \frac {q} {\Phi(q)}\frac{ \mathrm{i}\theta + \Phi(q) }{\Psi(-\theta)+q},
\end{align*}
matching \eqref{mu_laplace_classical_SP} (with $\theta$ replaced by $- \mathrm{i} \theta$). Therefore the continuity theorem shows the weak convergence.

\end{proof}


\begin{remark}
(1) For Problem 1, the Laplace transform \eqref{laplace_transform_mu_SP} can be written
\begin{align*}
		\int_{[0,\infty)} e^{-\theta y}\mu_1^{\lambda}(\diff y)=\frac{\varphi_1(0,q,\lambda)}{\varphi_1(\theta,q,\lambda)}
		\quad \textrm{where} \; \varphi_1(\theta,q,\lambda)
&:= \frac {(\psi(\theta)-q)(\Phi(q+\lambda)-\theta)} {(\lambda + q-\psi(\theta) ) (\theta - \Phi(q)) }.
\end{align*}
(2)  For Problem 2, the Fourier transform \eqref{laplace_transform_mu_SP_2} can be written
\begin{align*}
	\int_{\R} e^{\mathrm{i}\theta y}\mu_2^{\lambda}(\diff y)=\frac{\varphi_2(0,q,\lambda)}{\varphi_2(\theta,q,\lambda)}
	\quad \textrm{where} \; 	\varphi_2(\theta,q,\lambda)
	:= 
\frac {(\Psi(-\theta)+q)(\Phi(q+\lambda)+\mathrm{i}\theta)} { \mathrm{i}\theta + \Phi(q) }.  
\end{align*}
These can be seen as generalizations of  \eqref{mu_laplace2}.
\end{remark}

\begin{remark}
The weak convergence as in Propositions \ref{convergence} and \ref{proposition_SP} implies the convergence of the Gittins index, if $R(\cdot)$ and $r(\cdot)$ are bounded and continuous in Problems 1 and 2, respectively.
\end{remark}


%
%

\section{Concluding remarks} \label{section_conclusion}

In this paper, we studied an extension of the MAB driven by \lev processes  \cite{KM95} where decision opportunities arrive  at Poisson arrival times. The Gittins index can be written analogously to the classical case \eqref{gittins}. In particular, we obtained explicitly its form in terms of the scale function for the case of spectrally one-sided \lev processes and showed its convergence to those in \cite{KM95} as the arrival rate increases to infinity. 

The studied problem is new to the best of our knowledge and there are several venues for future research. 

First, in the current paper, we did not discuss the optimality of the Gittins index policy, and it is of great importance to analyze its near or exact optimality.
The convergence results obtained in Section \ref{section_convergence} suggest that it is  near optimal at least when $\lambda$ is high. It is a non-trivial and crucial problem to investigate whether the optimality remains to hold generally under the Poissonian decision times.

Second, one potential application of the considered problem is to approximate the optimal solutions in the classical discrete-time model (in Section \ref{section_discrete_time}) in terms of the Gittins index policy in the considered problem by approximating deterministic decision times in terms of Poisson arrival times. 
  This is related to Carr's randomization \cite{Carr}, where constant discrete decision times are approximated by Erlang random variables, or the sum of independent exponential random variables. This approximation method is conjectured to be practical and accurate in view of the existing related results in  stochastic control  
 such as \cite{Leung}.

Third, the obtained results are conjectured to hold for more general processes. The convergence of our  Gittins index to that in \eqref{gittins} is easily conjectured to hold for a general \lev process with two-sided jumps. In addition, the Gittins index may admit analytical expression written in terms of the Wiener-Hopf factorization, even when the spectrally one-sided assumption is relaxed. It is also of great interest to generalize the results to a more general Markov process with jumps by taking advantage of the recent development of excursion theory such as \cite{Noba}.

\end{document}